\documentclass[11pt,twoside]{article}

\usepackage[a4paper,margin=3.5cm]{geometry}
\usepackage[utf8]{inputenc}
\usepackage[english]{babel}
\usepackage{csquotes}
\usepackage{amssymb}
\usepackage{amsmath}
\usepackage{amsthm}
\usepackage{bm}
\usepackage{dsfont}
\usepackage{enumitem}
\usepackage{epigraph}
\usepackage{fancyhdr}
\usepackage{mathrsfs}
\usepackage{mathtools}
\usepackage{graphicx}
\usepackage{array}
\usepackage{pgf,tikz}
\usepackage{tcolorbox}
\usepackage{tikz-cd}
\usepackage{titling}
\usepackage{url}
\usepackage{verbatim}
\usepackage{ytableau}
\usepackage[style=numeric-comp, backend=biber]{biblatex}
\usepackage{hyperref}
\hypersetup{
    colorlinks=true
}

\setlist{noitemsep}

\usetikzlibrary{arrows}
\usetikzlibrary{decorations.pathmorphing}
\usetikzlibrary{patterns}

\newcommand{\FF}{\mathbb{F}}

\newcommand{\GG}{\mathbb{G}}

\newcommand{\lra}{\longrightarrow}

\newcommand{\mc}{\mathcal}

\newcommand{\N}{\mathbb{N}}
\newcommand{\OO}{\mathcal{O}}
\newcommand{\op}{\tn{op}}

\renewcommand{\P}{\mathbb{P}}

\newcommand{\ra}{\rightarrow}

\newcommand{\stlra}[1]{\stackrel{#1}{\longrightarrow}}
\newcommand{\stackstag}[1]{\cite[\href{https://stacks.math.columbia.edu/tag/#1}{Tag #1}]{stacks-project}}
\newcommand{\T}{\mathcal{T}}
\newcommand{\tn}{\text}

\newcommand{\U}{\mc{U}}

\newcommand{\Z}{\mathbb{Z}}

\newcommand{\lo}[1]{\mathrm{L}^{\perp}{#1}}
\newcommand{\ro}[1]{\mathrm{R}^{\perp}{#1}}

\DeclareMathOperator{\Bgm}{\mathrm{B_{gm}}}

\DeclareMathOperator{\Ch}{Ch} 

\DeclareMathOperator{\cone}{cone}

\DeclareMathOperator{\End}{End}

\DeclareMathOperator{\Ext}{Ext}

\DeclareMathOperator{\GL}{GL}

\DeclareMathOperator{\Gr}{Gr}

\DeclareMathOperator{\GW}{GW}

\DeclareMathOperator{\Hom}{Hom}

\DeclareMathOperator{\K}{K}

\DeclareMathOperator{\Perf}{Perf}

\DeclareMathOperator{\sheafhom}{\mathscr{H}\textit{\kern -4pt om}\,}

\DeclareMathOperator{\SL}{SL}

\DeclareMathOperator{\Spec}{Spec}

\DeclareMathOperator{\Sym}{Sym}

\DeclareMathOperator{\Vect}{Vect}
\DeclareMathOperator{\W}{W}

\DeclareMathOperator{\Rep}{\mathbf{Rep}}

\numberwithin{equation}{section}

\theoremstyle{definition}
\newtheorem{definition}[equation]{Definition}
\newtheorem{example}[equation]{Example}

\newtheorem{remark}[equation]{Remark}

\theoremstyle{plain}
\newtheorem{corollary}[equation]{Corollary}
\newtheorem{lemma}[equation]{Lemma}
\newtheorem{proposition}[equation]{Proposition}
\newtheorem{theorem}[equation]{Theorem}

\title{\textsc{hermitian k-theory of grassmannians}}
\author{
    \textsc{herman rohrbach}
    \thanks{
        The author was partially supported by the research training group \emph{GRK 2240: Algebro-Geometric Methods in Algebra, Arithmetic and Topology} and by the ERC through the project QUADAG.
        This paper is part of a project that has received funding from the European Research Council (ERC) under the European Union's Horizon 2020 research and innovation programme (grant agreement No. 832833). \newline
        \includegraphics[scale=0.08]{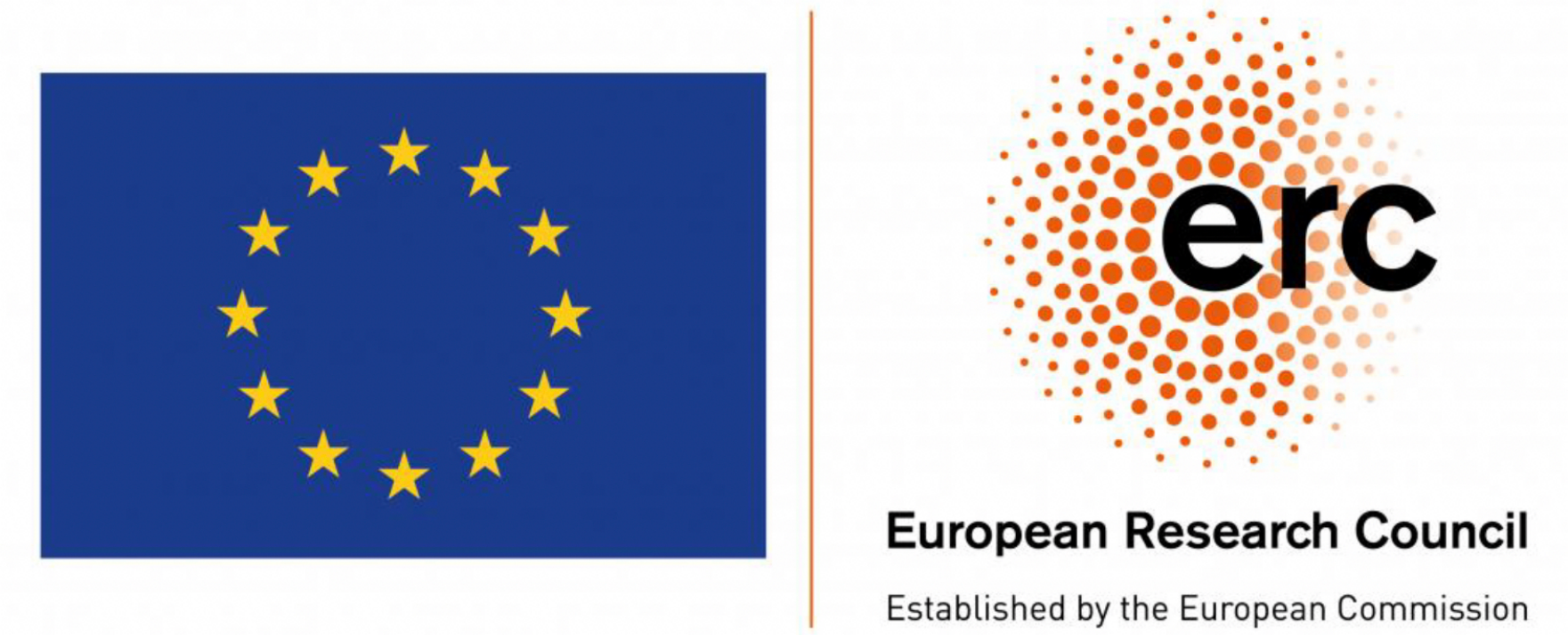}
    }
}
\predate{}
\postdate{}
\date{}

\begin{document}

\maketitle

\begin{abstract}
    We compute the additive structure of the Hermitian $\K$-theory spectrum of an even-dimensional Grassmannian over a base field $k$ of characteristic zero in terms of the Hermitian $\K$-theory of $X$, using certain symmetries on Young diagrams.
    The result is a direct sum of copies of the $\K$-theory of the base field and copies of the $\GW$-theory of the base field, indexed by \emph{asymmetric} and \emph{symmetric} Young diagrams, respectively.
\end{abstract}

\section{Introduction}
    \label{section:introduction}

In the context of algebraic geometry and motivic homotopy theory, Hermitian $\K$-theory is the study of vector bundles with quadratic forms on schemes. 
It is a finer invariant than algebraic $\K$-theory, which also makes it harder to compute. 
Many known computations in algebraic $\K$-theory have yet to be tackled for Hermitian $\K$-theory.
Even the foundations of Hermitian $\K$-theory lag behind those of algebraic $\K$-theory, despite recent strides \cite{calmes2020hermitian1, calmes2020hermitian2, calmes2020hermitian3}.
Especially in equivariant Hermitian $\K$-theory, a number of analogues of foundational results in equivariant algebraic $\K$-theory are missing, notably dévissage, making computations notoriously difficult.

In their paper \cite{balmer12grassmann} on Witt groups of Grassmannians, Balmer and Calmès remark that it might be surprising to some that it took almost thirty years after the invention of Witt groups to do the computation. 
In the same spirit, it might be surprising that it took another decade to extend their results to higher Hermitian $\K$-theory, even partially, since the fundament of the theory has been laid in e.g. \cite{balmer99witt1, balmer99witt2, hornbostel04, hornbostel05}.

This paper continues the work begun in \cite{rohrbach22} by providing key computations of the Hermitian $\K$-theory of Grassmannians over a base field $k$ of characteristic zero.
The strategy for these computations depends on certain symmetries on Young diagrams; as far the author is aware, it is a novelty to consider these symmetries as they relate to Hermitian $\K$-theory, though there might be a connection with the method of computation used in \cite{balmer12grassmann}.

For a Grassmannian $X = \Gr(d, d+e)$ of even dimension $de$, there is an involution on the set of Young diagrams, which corresponds to the action of the standard dualiy on a certain semi-orthogonal decomposition of the derived category of $X$, which is the key observation in the computation.
The main result of this paper is Theorem \ref{theorem:grothendieckwittspectrumofevengrassmannian}, which gives a combinatorial description of the additive structure of the Grothendieck-Witt spectrum of an even-dimensional Grassmannian in terms of copies of the Grothendieck-Witt spectrum and the algebraic $\K$-theory spectrum of the base field $k$.

\subsection*{Acknowledgements}
    \label{subsection:acknowledgements}
    
I am grateful to my advisors Jens Hornbostel and Marcus Zibrowius for numerous helpful discussions.

\section{Semi-orthogonal decompositions of derived categories of representations}
    \label{section:semiorthogonaldecompositionsofderivedcategoriesofrepresentations}

Before tackling Grassmannians, we study the representation theory of linearly reductive groups and the corresponding equivariant Grothendieck-Witt theory.
Throughout this section, the notation of \cite{conrad14} will be used.
Recall that an algebraic group $G$ is \emph{linearly reductive} if it is reductive and every representation splits into a direct sum of irreducible representations.
Throughout this section, $G$ will be a {linearly reductive group} over a field $k$, which encompasses general linear groups, special linear groups and other important examples in characteristic zero, and split tori in arbitary characteristic.
The representation theory of linearly reductive groups is well-understood, and we will apply this understanding to $G$-equivariant Grothendieck-Witt theory.
To this end, it is necessary to determine how the standard semi-orthogonal decomposition of $\Perf^G(k)$ behaves under duality.

Let $k$ be a field of characteristic zero.
Let $(G, B, T)$ be a triple consisting of a connected split reductive group $G$ over $k$ with $T \subset B \subset G$, where $B$ is a Borel subgroup and $T$ is a maximal torus of rank $t$. 
Let $X(T) \cong \Z^t$ be the character lattice of $T$.
Let $\Phi = \Phi(G,T)$ be a root system and let $\Phi^+ = \Phi(B,T)$ be the system of positive roots associated to $B$ as in \cite[Proposition 1.4.4]{conrad14} and let $\Delta \subset \Phi^+$ be the set of simple positive roots.
Let $W = W_G(T) = W(\Phi)$ be the Weyl group, with longest element $w_0 \in W$, which is the unique element such that $w_0(\Phi^+) = -\Phi^+$.
Finally, let
\begin{equation*}
    C = \left\{\lambda \in X(T) \mid \langle \lambda, a^{\vee} \rangle \geq 0 \tn{ for all }a \in \Delta \right\}
\end{equation*}
be a closed Weyl chamber, that is, a set of dominant weights.
The orbit of $C$ under the action of $W$ is $X(T)$. 

\begin{proposition} \label{proposition:semiorthogonaldecompositionequivariantperfectcomplexes}
The pretriangulated dg category $\mc{A} = \Perf^G(k)$ admits a semi-orthogonal decomposition
\begin{equation*}
    \mc{A} = \langle \mc{A}_{\lambda} \mid \lambda \in C \rangle,
\end{equation*}
where $\mc{A}_{\lambda} = \langle M_{\lambda} \rangle$ is generated by the irreducible representation $M_{\lambda}$ of highest weight $\lambda$.
Furthermore, $\mc{A}_{\lambda} \simeq \Perf(k)$ as dg categories.
\end{proposition}

\begin{proof}
By the proof of \cite[Lemme 6]{serre68} (or the theorem of the highest weight \cite[Theorem 1.5.6]{conrad14}), isomorphism classes of irreducible representations of $G$ correspond bijectively to elements of $C$.
Therefore, let $\{M_{\lambda} \mid {\lambda} \in C\}$ be a set of representatives, and $\mc{A}_{\lambda} = \langle M_{\lambda} \rangle$ the full pretriangulated subcategory of $\mc{A}$ generated by $M_{\lambda}$.
Let
\begin{equation*}
    F_{\lambda}: \Perf(k) \lra \mc{A}_{\lambda}
\end{equation*}
be given by $F_{\lambda}(N) = N \otimes_k M_{\lambda}$, where $N$ is equipped with the trivial $G$-action.
The proof of \cite[Lemme 5]{serre68} shows that $M_{\lambda}$ is absolutely irreducible, so $\End^G_k(M_{\lambda}) \cong k$. 
It follows that $F_{\lambda}$ is fully faithful and essentially surjective.
Because $M_{\lambda}$ is flat, $F_{\lambda}$ is also exact.
Hence $\mc{A}_{\lambda} \simeq \Perf(k)$ as pretriangulated dg categories.

Let $\lambda, \mu \in C$ such that $\lambda \neq \mu$ and let $f: M_{\lambda} \ra M_{\mu}$ be a morphism in $\mc{A}$. 
By Schur's lemma $f=0$, since $M_{\lambda}$ and $M_{\mu}$ are not isomorphic.
Hence $\Hom(\mc{A}_{\lambda}, \mc{A}_{\mu}) = 0$.
Furthermore, since every representation of $G$ splits into irreducible representations, $\Ext^i(M_{\lambda}, M_{\mu}) = 0$ for all $i > 0$. 

It remains to be shown that the $\mc{A}_{\lambda}$ generate $\mc{A}$.
Let $M$ be a finite dimensional $G$-representation.
Then each irreducible $G$-representation in the decomposition of $M$ is contained in one of the $\mc{A}_{\lambda}$, whence $M \in \langle \mc{A}_{\lambda} \mid \lambda \in C \rangle$, which finishes the proof.
\end{proof}

\begin{example} \label{example:nontrivialextensiongenerallineargroupofdegreetwo}
Here is an example (c.f. \cite[Example 1.2]{buchweitz15}) that shows that the group $\Ext^1(A,B)$ does not always vanish for irreducible representations $A$ and $B$ of $G$ when $k$ is a field of nonzero characteristic.
Let $k = \FF_2$, $G = \GL_2$, and let $V$ be the standard two-dimensional representation of $G$.
Then $\Lambda^2 V$ and $\Sym^2 V$ are non-isomorphic irreducible $G$-representations, and the exact sequence
\begin{equation*}
    0 \lra \Lambda^2 V \lra V \otimes V \lra \Sym^2 V \lra 0
\end{equation*}
does not split.
Hence $\Ext^1(\Sym^2 V, \Lambda^2 V) \neq 0$. 
\end{example}

The non-vanishing of Ext-groups of irreducible representations is the most important obstruction to a proof of proposition \ref{proposition:semiorthogonaldecompositionequivariantperfectcomplexes} in arbitrary characteristic.

\begin{corollary} \label{corollary:equivariantktheoryreductivegroup}
The $G$-equivariant $\K$-theory of $k$ is given by
\begin{equation*}
    \K^G_i(k) \cong \bigoplus_{\lambda \in C} K_i(k) \cong K^G_0(k) \otimes_{\Z} K_i(k)
\end{equation*}
for all $i \in \N$. 
\end{corollary}

\begin{proof}
This follows from additivity for $\K$-theory \cite[Proposition~7.10]{blumberg13} and the semi-orthogonal decomposition of proposition \ref{proposition:semiorthogonaldecompositionequivariantperfectcomplexes}. 
\end{proof}

Recall the theorem of the highest weight \cite[Theorem 1.5.6]{conrad14}, which states that every irreducible $G$-representation has a unique highest weight.
We want to describe the duality on $\Perf^G(k)$, as well as its symmetric objects.
The following lemma is folklore.

\begin{lemma} \label{lemma:weightofdualrepresentation}
Let $V$ be an irreducible $G$-representation with highest weight $\lambda$.
Then the dual representation $V^{\vee}$ has highest weight $-w_0\lambda$. 
\end{lemma}

\begin{proof}
Let $\Omega_V \subset X(T)$ be the set of weights of $V$.
Note that $\Omega_{V^{\vee}} = -\Omega_{V}$.
As $\lambda$ is the highest weight of $V$ and $\Omega_{V^{\vee}}$ is $W$-invariant, all weights in $\Omega_{V^{\vee}}$ are of the form
\begin{equation*}
    - w_0\lambda + \sum_{b \in w_0(\Delta)} m_b b
\end{equation*}
with $m_b \in \Z_{\geq 0}$. 
Note that $w_0(\Delta) \subset -\Phi^+$, so for $b \in w_0(\Delta)$ and $m_b \in \Z_{\geq 0}$,
\begin{equation*}
    m_b b = - \sum_{a \in \Delta} n_a a
\end{equation*}
with $n_a \in \Z_{\geq 0}$, as $\Delta$ is a base for $\Phi^+$. 
Hence all weights in $\Omega_{V^{\vee}}$ are of the form
\begin{equation*}
    - w_0\lambda - \sum_{a \in \Delta} n_a a,
\end{equation*}
and it follows that $-w_0\lambda$ is the highest weight of $V^{\vee}$, as was to be shown.
\end{proof}

The following corollary is a direct consequence of the above lemma.

\begin{corollary} \label{corollary:actionofdualityonorthogonaldecompositionofrepresentations}
The duality functor $*: \mc{A} \ra \mc{A}$ on $\mc{A} = \Perf^G(k)$ sends $\mc{A}_a$ to $\mc{A}_{-w_0a}$, where $a \in C$.
\end{corollary}

\begin{definition} \label{definition:selfdualrepresentation}
We call a representation $E$ of $G$ \emph{self-dual} if there exists an isomorphism $\phi: E \ra E^{\vee}$. 
We call a self-dual $E$ \emph{symmetric} (resp. \emph{anti-symmetric}) if $\phi$ can be chosen to be a symmetric (resp. anti-symmetric) form.
Note that $E$ may be both symmetric and anti-symmetric.
\end{definition}

The following lemma combines \cite[Lemma 1.21]{calmes05} and \cite[Proposition 2.2]{calmes05}.

\begin{lemma} \label{lemma:selfdualsimpleissymmetricorantisymmetric}
Let $M$ be an irreducible representation of $G$. 
If there exists an isomorphism $\phi: M \ra M^{\vee}$, then $\phi$ is either symmetric or anti-symmetric and any isomorphism $M \ra M^{\vee}$ is a scalar multiple of $\phi$.
\end{lemma}

By \cite[Corollary 2.3]{calmes05}, a self-dual irreducible representation is always either symmetric or anti-symmetric, but never both.
By lemma \ref{lemma:weightofdualrepresentation}, the following definition is sensible.

\begin{definition} \label{definition:self-dualcharacter}
A character $x \in X(T)$ is called \emph{self-dual} if $x = -w_0x$, where $w_0$ is the longest element of the Weyl group $W$. 
\end{definition}

The following proposition is a reformulation of \cite[Proposition 4.4]{hemmert22} in the language of algebraic geometry.

\begin{proposition} \label{proposition:signofrepresentationbyparity}
Let $V$ be a self-dual $G$-representation of highest weight $\lambda$.
Let $2\rho^{\vee}$ be the sum of the positive coroots.
Then $V$ is symmetric if and only if $\langle 2\rho^{\vee}, \lambda \rangle$ is even and anti-symmetric if and only if $\langle 2\rho^{\vee}, \lambda \rangle$ is odd.
\end{proposition}

\begin{proof}
Let $\phi: \SL_2 \ra G$ be a homomorphism that restricts to $2\rho^{\vee}: \GG_m \ra T$ on the respective maximal tori $\GG_m \subset \SL_2$ and $T \subset G$ (cf. \cite[Theorem 1.2.7, Definition 1.2.8]{conrad14}). 
Since $\phi^*: \Rep(G) \ra \Rep(\SL_2)$ preserves the sign of a representation, all of the irreducible components of $\phi^*V$ that occur with odd multiplicity must have the same sign, and it suffices to find a single irreducible component of $\phi^*V$ with odd multiplicity.
Indeed, a direct sum of an even number of isomorphic irreducible components can be equipped with both a symmetric and an anti-symmetric form, and therefore tells us nothing about the sign of a representation.
The $\lambda$-weight space $W \subset V$ is one-dimensional by the theorem of the highest weight \cite[Theorem 1.5.6]{conrad14}.
Let $\omega$ be the fundamental weight of $\SL_2$. 
Then $\phi^*W \subset \phi^*V$ is a one-dimensional weight space for $\langle 2\rho^{\vee}, \lambda \rangle \omega$.
If $W' \subset V$ is a weight space for another weight $\lambda'$, then $\lambda' = \lambda - \beta$ for some positive root $\beta \in \Phi^+$, again by the theorem of the highest weight.
By construction, $\langle 2\rho^{\vee}, \alpha \rangle = 2$ for all simple roots $\alpha$, so in particular $\langle 2\rho^{\vee}, \beta \rangle > 0$.
Thus $\phi^*W'$ is a weight space for the weight
\begin{equation*}
    (\langle 2\rho^{\vee}, \lambda \rangle - \langle 2\rho^{\vee}, \beta \rangle)\omega < \langle 2\rho^{\vee}, \lambda \rangle \omega,
\end{equation*}
and it follows that the decomposition of $\phi^*V$ into irreducible representations contains precisely one copy of the irreducible $\SL_2$-representation of weight $\langle 2\rho^{\vee}, \lambda \rangle \omega$, which is symmetric if and only if $\langle 2\rho^{\vee}, \lambda \rangle$ is even and anti-symmetric if and only if $\langle 2\rho^{\vee}, \lambda \rangle$ is odd. 

It follows that $V$ is symmetric if and only if $\langle 2\rho^{\vee}, \lambda \rangle$ is even and anti-symmetric if and only if $\langle 2\rho^{\vee}, \lambda \rangle$ is odd, as was to be shown.
\end{proof}

Let $C_0 \subset C$ be the set of dominant weights fixed by $-w_0$.
By lemma \ref{lemma:selfdualsimpleissymmetricorantisymmetric}, $C_0 = C_0^+ \sqcup C_0^-$, where $C_0^+$ is the set of weights $\lambda$ such that $M_{\lambda}$ admits a symmetric form, and $C_0^-$ the set of weights $\lambda$ such that $M_{\lambda}$ admits an anti-symmetric form.
Let $D$ be a set of representatives for the set of orbits $(C - C_0)/(-w_0)$ of elements of $C$ that are not fixed by $-w_0$. 

It follows that there are three possibilities for the sign of an irreducible representation $M_{\lambda}$; it is either symmetric, anti-symmetric, or not self-dual at all. 
If $M_{\lambda}$ is anti-symmetric and we equip it with a specific anti-symmetric form $\phi_{\lambda}: M_{\lambda} \ra M_{\lambda}^{\vee}$, this yields an element $[M_{\lambda}, \phi_{\lambda}] \in \GW^{[2]}_0(\Perf^G(k))$. 
Hence
\begin{equation} \label{equation:antisymmetricgw}
    \GW^{[n]}(\mc{A}_{\lambda}) \simeq \GW^{[n+2]}(\Perf^G(k))
\end{equation}
for $\lambda$ of sign $-1$. 

\begin{proposition} \label{proposition:equivariantgwtheoryreductivegroup}
Let $n \in \Z$.
With notation as before, the $G$-equivariant $\GW$-theory of $k$ is given by
\begin{equation*}
    \GW^{[n]}_{G,i}(k) \cong \bigoplus_{\lambda^+ \in C_0^+} \GW^{[n]}_i(k) \oplus \bigoplus_{\lambda^- \in C_0^-} \GW^{[n+2]}_i(k) \oplus \bigoplus_{\lambda \in D} \K_i(k).
\end{equation*}
\end{proposition}

\begin{proof}
This follows from corollary \ref{corollary:actionofdualityonorthogonaldecompositionofrepresentations}, equation (\ref{equation:antisymmetricgw}) and additivity for the $\GW$-spectrum \cite[Proposition~6.8]{schlichting17}.
\end{proof}

\begin{example} \label{example:representationtheoryofgenerallineargroup}
Let $G = \GL_n$. 
The maximal torus of diagonal matrices has weight lattice $\Z^n$ with standard unit vectors $e_i$.
The roots $\Phi = \Phi(G,T)$ of $G$ are all elements of the form $e_i - e_j$ with $1 \leq i,j \leq n$ and $i \neq j$. 
Letting $B \subset G$ be the Borel subgroup of upper triangular matrices, the positive roots $\Phi^+ = \Phi(B,T)$ are all elements $e_i - e_j$ with $i < j$, and the simple roots $\Delta$ are the elements $e_i - e_{i+1}$. 
The Weyl group $W$ is generated by the reflections $s_{ij}: \Z^n \ra \Z^n$ that switch $e_i$ and $e_j$. 
The longest element $w_0 \in W$ with respect to $\Phi^+$ is the product
\begin{equation*}
    w_0 = \prod_{i < j} s_{ij}.
\end{equation*}
For example, if $n = 3$, then $w_0 = s_{12}s_{13}s_{23} = s_{13}$, and if $n = 4$, then $w_0 = s_{14}s_{23}$.
In general, there is an isomorphism $W \ra S_n$, $s_{ij} \mapsto (i~j)$, and under this isomorphism
\begin{equation*}
    w_0 = (1~n)(2~n-1)\cdots
\end{equation*}
is the order reversing permutation.
If $n$ is odd, there is a middle number which is fixed by $w_0$. 
The dominant weights of $G$ are given by
\begin{equation*}
    C = \left\{\sum_{i=1}^n m_ie_i \mid m_i \in \Z, m_i \geq m_{i+1}\right\},
\end{equation*}
which we can think of as \emph{partitions} when $m_1 \geq 0$.
This provides a useful connection to Young diagrams that we will use in Section \ref{section:hermitianktheoryspectraofgrassmannians}.
The $-w_0$-fixed points $C_0 \subset C$ are weights of the form
\begin{equation*}
    \lambda_0 = \sum_{i=1}^{\lceil n/2 \rceil}m_i e_i - \sum_{i=\lfloor n/2 \rfloor + 1}^{n}m_{n+1-i} e_i, 
\end{equation*}
still satisfying $m_i \geq m_{i+1}$.
Note that the positive coroots are elements of the form $e_i^{\vee} - e_j^{\vee}$, so the sum $2\rho^{\vee}$ of the positive coroots is 
\begin{equation*}
    \sum_{i<j}e_i^{\vee} + e_j^{\vee} = \sum_{i=1}^{\lfloor n/2 \rfloor}2ni.
\end{equation*}
In particular, the sum is even, so for any self-dual dominant weight $\lambda$, $\langle 2\rho^{\vee}, \lambda \rangle$ is even.
Hence, by Proposition \ref{proposition:signofrepresentationbyparity} all the self-dual irreducible representations of $G$ are symmetric, i.e. $C_0 = C_0^+$.
\end{example}

\section{Hermitian K-theory spectra of Grassmannians}
    \label{section:hermitianktheoryspectraofgrassmannians}

After computing the cohomology of projective spaces, one can take a few different directions for further computations in algebraic geometry, but perhaps the most natural generalization is to consider Grassmannians. 
Projective spaces are themselves edge cases of Grassmannians, and their combinatorics generalizes to other Grassmannians through the theory of \emph{Young diagrams}.

In this chapter, we compute the Grothendieck-Witt spectra of a certain class of Grassmannians.
This will in turn enable us to compute the higher Grothendieck-Witt groups of the geometric classifying spaces $\Bgm{\GL_n}$ of general linear groups, which is an important step in proving a general Atiyah-Segal completion theorem for Grothendieck-Witt theory.
    
\subsection{Semi-orthogonal decompositions for Grassmannians}

Let $k$ be a field of characteristic zero. 
Let $n = d + e$ be some positive integer.
Let $X = \Gr(d,n)$ be the Grassmannian of $d$-dimensional subspaces in $k^n$, with tautological bundle $\U$ of rank $d$ and dual bundle $\T$ of rank $e$.
These bundles fit together in an exact sequence
\begin{equation*}
    0 \lra \T \lra \OO_X^{\oplus n} \lra \U \lra 0.
\end{equation*}
Let $\Delta = \det \U$ be the determinant of the tautological bundle.

We can replace $\Spec k$ and $k^n$ with a base scheme $S$ over $k$, and $k^n$ with a vector bundle $\mc{V}$ over $S$.
For now, we refrain from stating results in such generality to ensure a clear exposition.

We will follow \cite{buchweitz15}, and eventually expand on the results obtained there by introducing duality.
Since we are working in characteristic zero, we could have followed \cite{kapranov88} instead.
Write $X$ as the homogeneous space $G/P$, where $G = \GL_n$ and $P$ is the parabolic subgroup
\begin{equation*}
    \left(
        \begin{array}{cc}
            \GL_d & * \\
            0 & \GL_e 
        \end{array}
    \right),
\end{equation*}
which contains both the Levi subgroup $H = \GL_d \times \GL_e$ and the Borel subgroup $B$ of upper triangular matrices.
As usual, let $T \subset B$ be the maximal split torus of diagonal matrices.
From now on, we will write $G_1 = \GL_d$ and $G_2 = \GL_e$, as well as $T_i = G_i \cap T$ for $i = 1,2$. 
The theory of vector bundles on $X$ is intimately related to the representation theory of $G_1$, which we will exploit.
The character lattice of $T_1 = T \cap G_1$ is $\Z^d$, and its positive roots are of the form
\begin{equation*}
    (0, \ldots, 0, 1, 0, \ldots, 0, -1, 0 \ldots, 0).
\end{equation*}
There is a natural partial order $\leq$ on $T_1$ for which $\lambda \leq \mu$ if and only if $\mu - \lambda$ is a sum of positive roots. 
The dominant weights $\alpha$ of $G_1$ are non-increasing tuples $(\alpha_1, \dots, \alpha_d) \in \Z^d$. 
Dominant weights with non-negative entries correspond to \emph{partitions}, which can be visualized as Young diagrams.
The Young diagram of a partition $(\alpha_1, \dots, \alpha_d)$ has $d$ rows of blocks, where the $i$-th row consists of $\alpha_i$ blocks, see figure \ref{figure:youngdiagram} for an example.

\begin{figure}
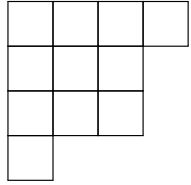

    \begin{equation*}
        \ydiagram{4,3,3,1}
    \end{equation*}    
    \caption{Young diagram for the partition $(4,3,3,1)$.}
    \label{figure:youngdiagram}
\end{figure}

Young diagrams, in turn, correspond to \emph{Schubert cells}, which are subschemes of Grassmannians that form cellular decompositions, and the intersection theory of these Schubert cells is called \emph{Schubert calculus}.
The Young diagrams considered here do \emph{not} correspond to the usual Schubert cells, but to twisted versions of these; here, the generator of $\mc{A}$ corresponding to a partition $\alpha$ is the Schur functor of $\alpha$ applied to the standard representation of $\GL_d$, twisted by a power of the determinant of the standard representation, as we will see later.

The \emph{degree} $|\alpha|$ of a weight $\alpha \in \Z^d$ is the sum $\sum \alpha_i$ of its entries. 
A representation has degree $m$ if all its weights have degree $m$, and a representation is called \emph{polynomial} if all its weights are partitions.
Let $V$ be the standard representation of $G_1$ of dimension $d$.
The determinant $\det V = \Lambda^d V$ of $V$ is an irreducible representation with highest weight $(1,\dots,1)$. 
For a partition $\alpha$ and a vector space (or $G_1$-representation) $W$, let 
\begin{equation*}
    \Lambda^{\alpha} W = \bigotimes_{i = 1}^d \Lambda^{\alpha_i} W,
\end{equation*}
and similarly when replacing $W$ with a vector bundle $\mc{E}$ on $X$.
Furthermore, each partition $\alpha$ allows the definition of the \emph{Schur functor $L^{\alpha}$} and the \emph{Weyl functor $K^{\alpha}$}, both of which are functors $\Vect(k) \ra \Rep(G_1)$. 

In particular, $L^{\alpha}V$ is the induced representation with highest weight $\alpha$, and since $k$ has characteristic zero, it is irreducible.
For the explicit constructions of the Schur and Weyl functors, see \cite[section 2]{buchweitz15}.
There is an equivalence of categories $\mc{L}_X: \Rep(P) \ra \Vect^G(X)$ by \cite[theorem 2.7]{cline83}.
Composing $\mc{L}_X$ with the inclusion $\Rep(G_1) \subset \Rep(P)$ induced by the projection $P \ra G_1$ allows us to study $G$-equivariant vector bundles on $X$ in terms of $G_1$-representations, for which we can use Example \ref{example:representationtheoryofgenerallineargroup}.

In representation theory and Schubert calculus, it seems to be customary to focus on polynomial representations, which is often justified by the remark that representations of $\GL_d$ can be tensored with the determinant representation with highest weight $(1,\dots, 1)$ to obtain polynomial representations, but the dual of a polynomial representation need not be polynomial.
We explain here how the determinant representation interacts with duality and how we can use this to obtain a satisfying theory of duality on Young diagrams, which makes a computation of the Grothendieck-Witt theory of Grassmannians possible.

Let $P_{d,e} \subset \Z^d$ be the set of partitions whose Young diagrams have at most $d$ rows and at most $e$ columns.
Such partitions correspond to the Schubert cells of $X$ and are visualized by Young diagrams that fit in a $(d \times e)$-frame.
For $\alpha \in P_{d,e}$, write $\alpha = (\alpha_1, \dots, \alpha_d)$ with $\alpha_{i} \geq \alpha_{i+1}$ for all $i=1, \dots, d-1$. 
For any $m \in \Z$ and $\alpha \in P_{d,e}$, we write $\underline{m}+\alpha$ for $(m+\alpha_1, \dots, m+\alpha_d) \in P_{d,e+m}$.

\begin{definition} \label{def:transposeofpartition}
Let $\alpha \in P_{d,e}$ be a partition.
The \emph{transpose $\alpha^{T}$ of $\alpha$} is the partition $\alpha^{T} \in P_{e,d}$ obtained by transposing the Young diagram of $\alpha$. 
\end{definition}

Since the Grassmannian $X$ is a projective scheme with an ample line bundle, we define the pretriangulated dg category
\begin{equation*}
    \mc{A} = \Perf(X) \simeq \Ch^b(\Vect(X))
\end{equation*}
with quasi-isomorphisms as its weak equivalences.
Note that the homotopy category $H^0\mc{A}$ is the usual bounded derived category $D^b(X)$ of $X$. 
For $i \in \N$, we let $\mc{C}_i \subset \Rep(G_1)$ be the full subcategory of the category of finite dimensional polynomial $G_1$-representations consisting of those representations whose irreducible components have highest weight $\alpha \in P_{d,e}$ such that $|\alpha| = i$. 
Let $(P_{d,e})_i \subset P_{d,e}$ be the subset of partitions of degree $i$. 
The bounded derived dg category $\Perf(\mc{C}_i)$ is generated as a pretriangulated dg category by the irreducible representations with highest weight $\alpha \in P_{d,e}$ with $|\alpha| = i$. 
By pulling back $G_1$-representations along the projection $P \ra G_1$, we can define, for each $i \in \N$, a fully faithful dg functor 
\begin{equation} \label{equation:glrepresentationtobundleongrassmanniann}
    \Phi_i: \Perf(\mc{C}_i) \ra \mc{A}    
\end{equation}
which sends $M$ to $\mc{L}_X(M)$, see \cite[p. 6,11]{buchweitz15} and note that this can really be constructed as a dg functor.
By \cite[theorem 5.6]{buchweitz15}, there is a semi-orthogonal decomposition
\begin{equation} \label{equation:semiorthogonaldecompositiongrassmannian}
    \mc{A} = \langle \mc{A}_0, \dots, \mc{A}_{de} \rangle,
\end{equation}
where $\mc{A}_i$ is the quasi-essential image of the dg functor $\Phi_i$, which is the full dg subcategory of $\mc{A}$ on objects quasi-isomorphic to objects in the image of $\Phi_i$. 
For each dominant weight $\lambda \in \Z^d$ of $G_1$, fix a representative $M_{\lambda}$ for the isomorphism class of irreducible representations of weight $\lambda$. 
Then $\Perf(\mc{C}_i)$ is generated by the irreducible representations $M_{\alpha}$ with highest weight $\alpha \in (P_{d,e})_i$.
For $\alpha, \beta \in (P_{d,e})_i$ such that $\alpha \neq \beta$, Schur's lemma ensures that $\Hom(M_{\alpha}, M_{\beta}) = 0$ in the category of $G_1$-representations.
Since the base field $k$ has characteristic zero, every $G_1$-representation splits into irreducible representations and it follows that 
\begin{equation*}
    \Ext^i(M_{\alpha},M_{\beta}) = \Hom(M_{\alpha}, M_{\beta}[i]) = 0    
\end{equation*}
in the triangulated category $H^0\Perf(\mc{C}_i)$. 
Hence, there is a semi-orthogonal decomposition
\begin{equation} \label{equation:semiorthogonaldecompositiontruncatedrepresentations}
    \Perf(\mc{C}_i) = \langle \langle M_{\alpha} \rangle \mid \alpha \in (P_{d,e})_i \rangle,
\end{equation}
where the set of partitions $P_{d,e}$ can be equipped with any linear order to obtain an semi-orthogonal decomposition of countably infinite length.
The semi-orthogonal decompositions (\ref{equation:semiorthogonaldecompositiongrassmannian}) and (\ref{equation:semiorthogonaldecompositiontruncatedrepresentations}) enable a computation of the $\K$-theory of $X$: by additivity for $\K$-theory
\begin{equation*}
    \K(\mc{A}) \simeq \bigoplus_{i=0}^{de} \K(\mc{A}_i) \simeq \bigoplus_{\alpha \in P_{d,e}} \K(k),
\end{equation*}
where the final equivalence follows from the fact that $\langle M_{\alpha} \rangle \simeq \Perf(k)$ for each $\alpha \in P_{d,e}$. 
Note that $|P_{d,e}| = \binom{n}{d}$.

\subsection{Duality on Young diagrams}
    \label{subsection:dualityonyoungdiagrams}

We would like to use the semi-orthogonal decomposition from the previous section for the computation of Witt groups and Grothendieck-Witt groups, but the standard duality on $\mc{A}$ does not permute the factors of the decomposition, and we cannot use the relevant additivity theorems.
We will now present a way to overcome this problem when either $d$ or $e$ is even, that is, when the Grassmannian is even-dimensional.
The inspiration for this solution came from a combination of the solution for projective spaces as presented in \cite{rohrbach22}, and the semi-orthogonal decomposition of proposition \ref{proposition:semiorthogonaldecompositionequivariantperfectcomplexes}, which \emph{is} permuted by the duality by corollary \ref{corollary:actionofdualityonorthogonaldecompositionofrepresentations}.

The case when both $d$ and $e$ are odd remains open, though the author expects the existence of an approach that unifies all the cases, in a similar fashion to \cite[theorem 6.1]{balmer12grassmann}, the results of which we will compare with that of our approach in due course.

First, we must understand the duality on $\mc{A}$.
Since, by \cite[lemma 5.1]{buchweitz15}, $\mc{A}$ is generated by exterior powers $\Lambda^{\alpha^{T}} \U$ with $\alpha \in P_{d,e}$ where $\alpha^{T} \in P_{e,d}$ is the transpose of $\alpha \in P_{d,e}$, it suffices to understand the duals of generators.

Recall from example \ref{example:representationtheoryofgenerallineargroup} that the longest element $w_0$ of the Weyl group of $\GL_e$ inverts the ordering of a dominant weight $\lambda \in \Z^e$, when seen as an ordered $e$-tuple. 
Also recall that $\U$ is the tautological bundle associated with the standard $d$-dimensional representation of $G_1$ and $\Delta = \det \U$ is the line bundle associated with the determinant representation with highest weight $(1,\dots,1)$.

\begin{lemma} \label{lemma:dualofexteriorpowerofstandardrepresentation}
Let $\alpha \in P_{d,e}$.
Then
\begin{equation*}
    (\Lambda^{\alpha^{T}}\U)^{\vee} \cong \Lambda^{\underline{d}-w_0\alpha^{T}}\U \otimes \Delta^{-e},
\end{equation*}
where $\Delta^{-e} = (\Delta^{\vee})^e$. 
\end{lemma}

\begin{proof}
Because the duality $\vee: \mc{A}^{\op} \ra \mc{A}$ commutes with tensor products and $\Lambda^{\alpha^{T}}\U$ is invariant (up to repeated switch isomorphisms) under permutations of the entries of $\alpha^{T}$, it suffices to show
\begin{equation*}
    (\Lambda^i \U)^{\vee} \cong \Lambda^{d-i}\U \otimes \Delta^{-1}
\end{equation*}
for $0 \leq i \leq d$, but this follows from the perfect pairing $\Lambda^i \U \otimes \Lambda^{d-i}\U \ra \Delta$: it gives an isomorphism
\begin{equation*}
    \Lambda^i \U \lra \Hom(\Lambda^{d-i}\U, \Delta),
\end{equation*}
and taking the dual on both sides gives the desired result.
\end{proof}

Inspired by the above lemma, we make a few definitions to help us organize our data.

\begin{figure}
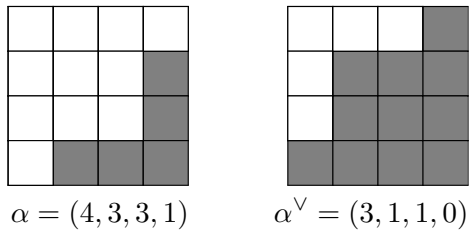

    \begin{equation*}
        \begin{aligned}
            \ydiagram
            [*(white)]{4,3,3,1}
            *[*(gray)]{4,4,4,4}
                & \qquad
                    & \ydiagram
                [*(white)]{3,1,1}
                *[*(gray)]{4,4,4,4} \\
            \alpha = (4,3,3,1)
                & \qquad
                    & \alpha^{\vee} = (3,1,1,0)
        \end{aligned}
    \end{equation*}   
    \caption{the dual of a Young diagram.}
    \label{figure:dualyoungdiagram}
\end{figure}

\begin{figure}
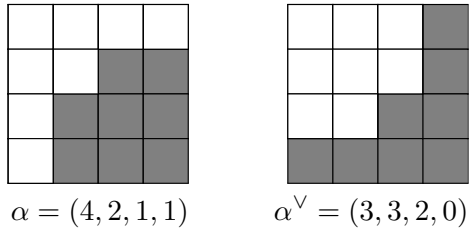

    \begin{equation*}
        \begin{aligned}
            \ydiagram
            [*(white)]{4,2,1,1}
            *[*(gray)]{4,4,4,4}
                & \qquad
                    & \ydiagram
                [*(white)]{3,3,2}
                *[*(gray)]{4,4,4,4} \\
            \alpha = (4,2,1,1)
                & \qquad
                    & \alpha^{\vee} = (3,3,2,0)
        \end{aligned}
    \end{equation*}   
    \caption{an example of a half partition and its dual.}
    \label{figure:halfpartition}
\end{figure}

\begin{definition} \label{def:dualofyoungdiagram}
Let $\alpha \in P_{d,e}$.
Let $w_0$ be the longest element of the Weyl group of $\GL_d$. 
The \emph{dual partition $\alpha^{\vee}$ of $\alpha$ in $P_{d,e}$} is the partition $\underline{e}-w_0\alpha$. 
Note that if $|\alpha| = i$, then $|\alpha^{\vee}| = de - i$. 
Therefore, we call $\alpha$ a \emph{half partition of $P_{d,e}$} if $\alpha$ has weight $de'$, so that $|\alpha| = |\alpha^{\vee}|$.
\end{definition}

Note that if $|\alpha| = i$, then $|\alpha^{\vee}| = de - i$, and that the degree of a weight is invariant under $w_0$.
Additionally, note that the dual of the transpose $\alpha^{T} \in P_{e,d}$ of some $\alpha \in P_{d,e}$ is $w_0(\underline{d}-\alpha^{T}) \in P_{e,d}$, where $w_0$ is the longest element of the Weyl group of $G_2$.
This duality on partitions is best described pictorially with Young diagrams, see figure \ref{figure:dualyoungdiagram}. 
The dual of a Young diagram is obtained by rotating the diagram $180$ degrees and swapping the filled part with the empty part of the $(d \times e)$-frame.
For an example of a half partition in a $4 \times 4$-frame, see figure \ref{figure:halfpartition}.

\subsection{The Grothendieck-Witt spectrum of an even Grassmannian}
    \label{subsection:thecomputationofthegrothendieckwittspectrumofagrassmannian}

We call the Grassmannian $X = \Gr(d, d+e)$ \emph{even} if its dimension $de$ is even.
From here on, we assume that $e$ is even (so the Grassmannian is even) and let $e' = e/2$.
We return to the semi-orthogonal decomposition of $\mc{A}$. 
The functor $- \otimes \Delta: \mc{A} \ra \mc{A}$ is an equivalence of pretriangulated dg categories.
Then we obtain the twisted semi-orthogonal decomposition
\begin{equation} \label{equation:shiftedsemiorthogonaldecompositionofgrassmannian}
    \mc{A} = \langle  \mc{A}_0 \otimes \Delta^{-e'}, \dots, \mc{A}_{de} \otimes \Delta^{-e'} \rangle
\end{equation}
by Corollary \ref{corollary:twistingsemiorthogonaldecompositionbyequivalence}.
The reason for this twist is explained by the following results.

\begin{lemma} \label{lemma:shifteddecompositionofgrassmannianstableunderduality}
The standard duality $\vee: \mc{A}^{\op} \ra \mc{A}$ acts on the semi-orthogonal decomposition (\ref{equation:shiftedsemiorthogonaldecompositionofgrassmannian}) as
\begin{equation*}
    (\mc{A}_i \otimes \Delta^{-e'})^{\vee} = \mc{A}_{de-i} \otimes \Delta^{-e'}
\end{equation*}
for each $i \in \{0,\dots,de\}$. 
\end{lemma}

\begin{proof}
Let $\alpha \in P_{d,e}$ be a partition of degree $i$.
Then $\Lambda^{\alpha^{T}}\U \otimes \Delta^{-e'}$ is a generator of $\mc{A}_i \otimes \Delta^{-e'}$.
Combining lemma \ref{lemma:dualofexteriorpowerofstandardrepresentation} with definition \ref{def:dualofyoungdiagram}, we obtain
\begin{equation*}
    (\Lambda^{\alpha^{T}}\U \otimes \Delta^{-e'})^{\vee} = \Lambda^{(\alpha^{T})^{\vee}}\U \otimes \Delta^{-e'},
\end{equation*}
which is a generator of $\mc{A}_{de-i} \otimes \Delta^{-e'}$ since $|(\alpha^{T})^{\vee}| = de-i$. 
\end{proof}

From now on, let $\mc{B}_i = \mc{A}_i \otimes \Delta^{-e'}$. 
As an immediate corollary, we obtain a partial calculation of the Grothendieck-Witt spectrum $\GW(X)$.
This calculation holds in arbitrary characteristic due to the results of \cite{buchweitz15}. 

\begin{corollary} \label{corollary:partialcomputationgrothendieckwittgroupofgrassmannian}
There is a natural equivalence
\begin{equation*}
    \GW^{[n]}(X) \simeq \GW^{[n]}(\mc{B}_{de'}) \oplus \bigoplus_{i=0}^{de' - 1} \K(\mc{B}_{i}).
\end{equation*}
In particular,
\begin{equation*}
    \W^{[n]}(X) \simeq \W^{[n]}(\mc{B}_{de'}).
\end{equation*}
\end{corollary}

\begin{proof}
This is an application of the additivity theorem for Grothendieck-Witt theory since $(\mc{B}_i)^{\vee} = \mc{B}_{de-i}$ for all $1 \leq i \leq de'$; in particular, $(\mc{B}_{de'})^{\vee} = \mc{B}_{de'}$.
\end{proof}

As $\K(\mc{B}_i)$ splits into copies of $\K(k)$, it remains to compute $\GW^{[n]}(\mc{B}_{de'})$.
For $0 \leq i \leq de$, define $\mc{D}_i = \Perf(\mc{C}_{i}) \otimes (\det V)^{-e'}$ as a pretriangulated dg subcategory of $\Perf(\Rep(G_1))$. 
By \cite[theorem 5.8]{buchweitz15} and \cite[section 3.1.4]{cline83}, the dg functor
\begin{equation*}
    \Phi': \mc{D}_{de'} \lra \mc{B}_{de'}
\end{equation*}
given by $M \otimes (\det V)^{-e'} \mapsto \mc{L}_X(M) \otimes \Delta^{-e'}$ is a duality-preserving equivalence.
The twisted version of the semi-orthogonal decomposition (\ref{equation:semiorthogonaldecompositiontruncatedrepresentations}) is
\begin{equation} \label{equation:semiorthogonaldecompositionmiddletwistedtruncatedrepresentations}
    \mc{D}_{de'} = \langle \langle M_{\alpha - \underline{e'}} \rangle \mid \alpha \in (P_{d,e})_{de'} \rangle,
\end{equation}
using the isomorphism $M_{\alpha} \otimes \det V^{-e'} \cong M_{\alpha - \underline{e'}}$ for $\alpha \in P_{d,e}$.
Note that this semi-orthogonal decomposition does not exist in prime characteristic, because the irreducible representations $M_{\alpha}$ do not form an exceptional collection, see example \ref{example:nontrivialextensiongenerallineargroupofdegreetwo}. 

\begin{lemma} \label{lemma:dualityofrepresentationsbyshiftedpartitions}
For a partition $\alpha \in P_{d,e}$, 
\begin{equation*}
    M_{\alpha - \underline{e'}}^{*} \cong M_{\alpha^{\vee} - \underline{e'}}.
\end{equation*}
\end{lemma}

\begin{proof}
By lemma \ref{lemma:weightofdualrepresentation}, 
\begin{equation*}
    M_{\alpha - \underline{e'}}^{*} \cong M_{-w_0(\alpha - \underline{e'})}.
\end{equation*}
Since $w_0(\underline{m}) = \underline{m}$ for all $m \in \Z$, $-w_0(\alpha - \underline{e'}) = -w_0\alpha + \underline{e'}$.
On the other hand, definition \ref{def:dualofyoungdiagram} yields
\begin{equation*}
    \alpha^{\vee} - \underline{e'} = \underline{e} - w_0\alpha - \underline{e'} = -w_0\alpha + \underline{e'},
\end{equation*}
which concludes the proof.
\end{proof}

In particular, if $\alpha$ is a half partition, then $M_{\alpha - \underline{e'}}$ and its dual have highest weights of the same degree.
Hence, to study how $\mc{D}_{de'}$ behaves under duality, we need to understand the combinatorics of half partitions.
We will distinguish between two kinds of half partitions.

\begin{definition} \label{def:symmetrichalfpartition}
A \emph{symmetric half partition $\alpha \in P_{d,e}$} is a half partition (see definition \ref{def:dualofyoungdiagram}) such that $\alpha^{\vee} = \alpha$.
A half partition is called \emph{asymmetric} if it is not symmetric.
The Young diagram of a symmetric partition is also called \emph{symmetric}.
\end{definition}

\begin{lemma} \label{lemma:numberofsymmetrichalfpartitions}
Let $d' = \lfloor \tfrac{d}{2} \rfloor$. 
There are $\binom{d' + e'}{e'}$ symmetric half partitions in $P_{d,e}$.
\end{lemma}

\begin{proof}
Note that Young diagrams in the $(d \times e)$-frame correspond to binary sequences of length $d+e$, containing $d$ zeroes and $e$ ones.
Given such a binary sequence, if we think of a Young diagram as a path starting in the lower left corner of the frame, each zero in the sequence means going up one row, and each one means going one column to the right.
For example, the Young diagram $(4,3,3,1)$ in the $(4 \times 4)$-frame corresponds to the binary sequence $10110010$, read from left to right.
A Young diagram is symmetric if and only if the corresponding binary sequence is a palindrome.
If $d$ is odd, it follows that the middle bit in the binary sequence of a symmetric Young diagram must be a one, since $e$ is even by assumption.
Thus a symmetric Young diagram is completely determined by the first $d' + e'$ bits of the corresponding binary sequence, $e'$ of which must be zero.
It follows that the number of different symmetric Young diagrams in the $(d \times e)$-frame is
\begin{equation*}
    \binom{d'+e'}{e'},
\end{equation*}
as was to be shown.
\end{proof}

The total number of asymmetric half partitions in $P_{d,e}$ is harder to determine, but can be expressed using a recursive formula.
Ultimately, we are more interested in the symmetric half partitions because the asymmetric half partitions only contribute to the hyperbolic part of the Grothendieck-Witt spectrum.
For $i \in \N$, $|(P_{d,e})_i|$ is the total number of partitions of degree $i$ in $P_{d,e}$.

\begin{lemma} \label{lemma:numberofhalfpartitions}
The total number of half partitions $|(P_{d,e})_{de'}|$ is even, and given by the recursive formula
\begin{equation*}
    |(P_{d,e})_{de'}| = \sum_{j=1}^{e} |(P_{d-1,j})_{de' - j}|.
\end{equation*}
\end{lemma}

\begin{proof}
An $\alpha \in (P_{d,e})_{de'}$ is given by a $d$-tuple $(\alpha_1, \dots, \alpha_d)$ such that $0 \leq \alpha_i \leq e$ and $\alpha_i \geq \alpha_{i+1}$ for all $0 \leq i \leq d$.
Note that $\alpha_1 \neq 0$ since $de' \neq 0$ by assumption. 
If $\alpha_1 = j$ for some $j \in \N$, then $\alpha_i \leq j$ for all $i \geq 2$. 
Hence $(\alpha_2, \dots, \alpha_d) \in (P_{d-1,j})_{de' - j}$, and the recursive formula follows.
\end{proof}

Finally, here is a computation of the Grothendieck-Witt spectrum of $\mc{D}_{de'}$. 
Let $A_{d,e}$ be the number of asymmetric half partitions of $P_{d,e}$. 

\begin{lemma} \label{lemma:grothendieckwittofmiddlecomponentofgrassmannian}
There is an equivalence
\begin{equation*}
    \GW^{[n]}(\mc{D}_{de'}) \simeq \bigoplus_{i=1}^{\binom{d'+e'}{e'}}\GW^{[n]}(k) \oplus \bigoplus_{j=1}^{\tfrac{A_{d,e}}{2}} \K(k).
\end{equation*}
\end{lemma}

\begin{proof}
On the one hand, the duality $\vee: (P_{d,e})_{de'} \ra (P_{d,e})_{de'}$ sends asymmetric half partitions to asymmetric half partitions, leaving no asymmetric half partition fixed.
Hence, if $\alpha$ is an asymmetric half partition, $M_{\alpha - \underline{e'}}$ is different from its dual and contributes a copy of $\K(\langle M_{\alpha - \underline{e'}} \rangle)$ to the direct sum decomposition of $\GW^{[n]}(\mc{D}_{de'})$ by additivity.

On the other hand, if $\alpha$ is a symmetric half partition, $M_{\alpha - \underline{e'}}$ is self-dual and contributes a copy of $\GW^{[n]}(\langle M_{\alpha - \underline{e'}} \rangle)$ to the direct sum decomposition.
By proposition \ref{proposition:semiorthogonaldecompositionequivariantperfectcomplexes}, $\langle M_{\alpha - \underline{e'}} \rangle \simeq \Perf(k)$, which finishes the proof.
\end{proof}

Putting everything together, we obtain our main result for Grothendieck-Witt spectra of Grassmannians.

\begin{theorem} \label{theorem:grothendieckwittspectrumofevengrassmannian}
Let $p = \binom{d'+e'}{e'}$ and $q = \tfrac{1}{2}(\binom{d+e}{e} - p)$.
There is an equivalence
\begin{equation*}
    \GW^{[n]}(X) \simeq \bigoplus_{i=1}^{p}\GW^{[n]}(k) \oplus \bigoplus_{j=1}^{q} \K(k).
\end{equation*}
\end{theorem}

\begin{proof}
By corollary \ref{corollary:partialcomputationgrothendieckwittgroupofgrassmannian}, copies of $\GW^{[n]}(k)$ are contributed solely by
\begin{equation*}
    \GW^{[n]}(\mc{D}_{de'}),
\end{equation*}
and lemma \ref{lemma:grothendieckwittofmiddlecomponentofgrassmannian} tells us that there exactly $p$ of such copies.

Furthermore, the exceptional collection of $\mc{A}$ consisting of bundles $L^{\alpha}\mc{U}$ has $\binom{d+e}{e}$ members, one for each partition in $P_{d,e}$. 
Only $p$ of these generators are self-dual, and the rest occur in dual pairs.
Each such dual pair contributes exactly one copy of $\K(k)$ by additivity for $\GW$ and the fact that $\langle M_{\alpha - \underline{e'}} \rangle \simeq \Perf(k)$. 
The result follows.
\end{proof}

\appendix 

\section{Semi-orthogonal decompositions}
    \label{section:semiorthogonaldecompositions}

In this appendix, we give a reminder on semi-orthogonal decompositions of pretriangulated dg categories.
Usually, the theory of semi-orthogonal decompositions is explained using triangulated categories, so this section also provides details on the translation between the two settings.
Recall that a pretriangulated dg category is a dg category whose homotopy category is triangulated.
Let $\mc{A}$ be a pretriangulated dg category.

\begin{definition} \label{def:leftorthogonaldgcategory}
Let $\mc{B} \subset \mc{A}$ be a (not necessarily full) dg subcategory of a dg category.
The \emph{left orthogonal $\lo{\mc{B}}$ (resp. the right orthogonal $\ro{\mc{B}}$) of $\mc{B}$} is the full dg subcategory consisting of the objects $A \in \mc{A}$ such that the chain complex $\mc{A}(A,B)$ (resp. $\mc{A}(B,A)$) is acyclic for all $B \in \mc{B}$.
\end{definition}
    
\begin{definition} \label{def:semiorthogonaldecompositiondgcategory}
Let $\mc{A}$ be a pretriangulated dg category.
A pair of full pretriangulated dg subcategories $\mc{A}_-, \mc{A}_+$ forms a \emph{semi-orthogonal decomposition $\langle \mc{A}_-, \mc{A}_+ \rangle $ of $\mc{A}$} if
\begin{enumerate} [label=(\roman*)]
    \item for all objects $A_- \in \mc{A}_-$ and $A_+ \in \mc{A}_+$, the mapping complex $\mc{A}(A_+,A_-)$ is acyclic (``there are no morphisms from right to left''); and
    \item for every object $A \in \mc{A}$, there exists a closed morphism $f: A_+ \ra A$ with $A_+ \in \mc{A}_+$ whose cone $A_- = \cone(f)$ is in $\mc{A}_-$.
\end{enumerate}
\end{definition}

The following lemma shows that a semi-orthogonal decomposition of a pretriangulated dg category is the same as a semi-orthogonal decomposition of the underlying triangulated category.

\begin{lemma} \label{lemma:semiorthogonaldecompositionadmitsadjoints}
Let $\mc{A}$ be a pretriangulated dg category with full pretriangulated dg subcategories $\mc{A}_-$ and $\mc{A}_+$.
The following are equivalent.
\begin{enumerate}[label=(\roman*)]
    \item The pair $\langle \mc{A}_-, \mc{A}_+ \rangle$ is a semi-orthogonal decomposition of $\mc{A}$.
    \item The inclusion $H^0\mc{A}_- \subset H^0\mc{A}$ admits a left adjoint and $\lo{\mc{A}}_- = \mc{A}_+$. 
    \item The inclusion $H^0\mc{A}_+ \subset H^0\mc{A}$ admits a right adjoint and $\ro{\mc{A}_+} = \mc{A}_-$. 
\end{enumerate}
\end{lemma}

\begin{proof}
We will prove only that (i) is equivalent to (ii), the proof that (i) is equivalent to (iii) being similar.

First note that $\mc{A}(A_+, A_-)$ is acyclic if and only if $H^0\mc{A}(A_+, A_-[n]) = 0$ for all $n \in \Z$, so $\mc{A}(A_+,A_-)$ being acyclic for all $A_+ \in \mc{A}_+$ and $A_- \in \mc{A}_-$ is equivalent to $H^0\mc{A}(A_+,A_-)$ being zero for all $A_+ \in \mc{A}_+$ and $A_- \in \mc{A}_-$, since $\mc{A}_-$ is pretriangulated.

Let $f: A \ra B$ be a morphism in $H^0\mc{A}$.
Since $\langle \mc{A}_-, \mc{A}_+ \rangle$ is a semi-orthogonal decomposition of $\mc{A}$, we obtain two exact triangles
\begin{equation*}
    \begin{tikzcd}[row sep=small]
        A_+ \arrow[r, "a_+"]
            & A \arrow[r, "a_-"] 
                & A_- \\
        B_+ \arrow[r, "b_+"]
            & B \arrow[r, "b_-"] 
                & B_-                
    \end{tikzcd}
\end{equation*}
in $H^0\mc{A}$.
Then there is a morphism of exact triangles
\begin{equation*}
    \begin{tikzcd}
        A_+ \arrow[r, "a_+"] \arrow[d]
            & A \arrow[r, "a_-"] \arrow[d, "b_-f"]
                & A_- \arrow[d, dashed, "f_-"]\\
        0 \arrow[r]
            & B_- \arrow[r, equal]
                & B_-
    \end{tikzcd}
\end{equation*}
in which the dotted arrow $f_-:A_- \ra B_-$ exists and is unique because of \stackstag{0FWZ}.
Thus we can define a functor $F: H^0\mc{A} \ra H^0\mc{A}_-$ by choosing an $A_- \in \mc{A}_-$ for each $A \in \mc{A}$, which is left adjoint to the inclusion $H^0\mc{A}_- \ra H^0\mc{A}$ by construction.

Conversely, suppose that $F: H^0\mc{A} \ra H^0\mc{A}_-$ is left adjoint to the inclusion $H^0\mc{A}_- \ra H^0\mc{A}$ and that $\lo{\mc{A}}_- = \mc{A}_+$. 
Let $A \in \mc{A}$.
Then there is a canonical morphism $A \ra F(A)$ in $H^0\mc{A}$, which we extend to an exact triangle $A \ra F(A) \ra B$. 
If $A \in \mc{A}_-$, then the canonical morphism $A \ra F(A)$ is an isomorphism, as the inclusion $H^0\mc{A}_- \ra H^0\mc{A}$ is fully faithful.
Hence, the image of the exact triangle $A \ra F(A) \ra B$ under $F$ is isomorphic to the exact triangle $F(A) = F(A) \ra 0$, so $F(B) = 0$.
This yields $H^0\mc{A}(B,A_-) = H^0\mc{A}_-(F(B), A_-) = 0$ by adjunction for any $A_- \in \mc{A}_-$, so it follows that $B \in \mc{A}_+$.
Rotation of triangles yields an exact triangle $B[1] \ra A \ra F(A)$, which is induced by a morphism $B[1] \ra A$ with $B[1] \in \mc{A}_+$, whose cone is in $\mc{A}_-$. 
\end{proof}

Recall that for a triangulated category $\mc{T}$ and a thick subcategory $\mc{T}'$, the \emph{Verdier quotient} $\mc{T}/\mc{T}'$ is the unique (up to unique equivalence) triangulated category such that any exact functor $F: \mc{T}' \ra \mc{S}$ with $\mc{T} \ra \mc{T}' \ra \mc{S}$ trivial factors uniquely through the canonical functor $\mc{T}' \ra \mc{T}'/\mc{T}$.
Recall furthermore that a sequence $\mc{T}_1 \ra \mc{T}_2 \ra \mc{T}_3$ of triangulated categories is called \emph{exact} if the composition is trivial, $\mc{T}_1 \ra \mc{T}_2$ is a fully faithful exact functor whose image is a thick subcategory of $\mc{T}_2$, and the induced functor $\mc{T}_2/\mc{T}_1 \ra \mc{T}_3$ is an equivalence of triangulated categories.

\begin{definition} \label{def:quasiexactsequenceofdgcategories}
A sequence 
\begin{equation*}
    \mc{A} \ra \mc{B} \ra \mc{C}
\end{equation*}
of pretriangulated dg categories is called \emph{quasi-exact} if the associated sequence
\begin{equation*}
    H^0\mc{A} \ra H^0\mc{B} \ra H^0\mc{C}
\end{equation*}
of triangulated categories is exact.
It is called \emph{split exact} if $H^0\mc{A} \ra H^0\mc{B}$ admits a right adjoint.
\end{definition}

Semi-orthogonal decompositions correspond to split exact sequences by lemma \ref{lemma:semiorthogonaldecompositionadmitsadjoints} and the following result.

\begin{lemma} \label{lemma:rightorthogonalofsemiorthogonaldecompositionisquotient}
Let $\langle \mc{A}_-, \mc{A}_+ \rangle$ be a semi-orthogonal decomposition of $\mc{A}$.
Then $H^0\mc{A}_-$ is equivalent as a triangulated category to the Verdier quotient $H^0\mc{A}/H^0\mc{A}_+$.
\end{lemma}

\begin{proof}
We must show that $H^0\mc{A}_-$ satisfies the universal property of the Verdier quotient.
Let $F: H^0\mc{A} \ra H^0\mc{A}_-$ be the left adjoint to the inclusion $H^0\mc{A}_- \subset H^0\mc{A}$, and for $A \in H^0\mc{A}$ let $A_- = F(A)$. 
Let $G: H^0\mc{A} \ra H^0\mc{B}$ be an exact functor such that each $A_+ \in H^0\mc{A}_+$ is mapped to the zero object in $H^0\mc{B}$.
Then a morphism $f:A \ra A'$ in $H^0\mc{A}$ can be completed to a morphism of exact triangles
\begin{equation*}
    \begin{tikzcd}
        A_+ \arrow[r] \arrow[d, dashed]
            & A \arrow[r] \arrow[d, "f"]
                & A_- \arrow[d, "F(f)"] \\
        A_+' \arrow[r]
            & A' \arrow[r]
                & A_-'.
    \end{tikzcd}
\end{equation*}
As $G(A_+) = 0$, the image of this morphism of triangles under $G$ is
\begin{equation*}
    \begin{tikzcd}
        0 \arrow[r] \arrow[d, equal]
            & G(A) \arrow[r] \arrow[d, "G(f)"]
                & G(A_-) \arrow[d, "G(F(f))"] \\
        0 \arrow[r]
            & G(A') \arrow[r]
                & G(A_-'),
    \end{tikzcd}
\end{equation*}
and since $G$ is exact, this is a morphism of exact triangles.
It follows that $G(A) \ra G(A_-)$ is an isomorphism for all $A \in H^0\mc{A}$.
Hence $G$ factors through $F$, as was to be shown.
\end{proof}

Now we expand the definition of a semi-orthogonal decomposition by allowing more components, which will make the concept more flexible in practice.

\begin{definition} \label{def:semiorthogonaldecompositionmorecomponents}
Let $\mc{A}$ be a pretriangulated dg category with full pretriangulated dg subcategories $\mc{A}_1, \dots, \mc{A}_n$ for some $n \in \N_{\geq 2}$.
For $0 \leq i \leq n$, Define $\mc{A}_{\leq i} = \langle \mc{A}_1, \dots, \mc{A}_i \rangle$ to be the smallest full pretriangulated dg subcategory containing each $\mc{A}_j$ for $j \leq i$, and define $\mc{A}_{\geq i}$ similarly.
Note that $\mc{A}_{\leq 0} = 0$. 
Then $\mc{A}_1, \dots, \mc{A}_n$ form a \emph{semi-orthogonal decomposition $\langle \mc{A}_1, \dots, \mc{A}_n \rangle$ of $\mc{A}$} if $\langle \mc{A}_{\leq i}, \mc{A}_{\geq i+1} \rangle$ is a semi-orthogonal decomposition (definition \ref{def:semiorthogonaldecompositiondgcategory}) of $\mc{A}$ for each $1 \leq i \leq n-1$. 
\end{definition}

\begin{remark} \label{remark:filtrationinducedbysemiorthogonaldecomposition}
Unpacking definition \ref{def:semiorthogonaldecompositionmorecomponents}, we see that subcategories $\mc{A}_1, \dots, \mc{A}_n$ of $\mc{A}$ form a semi-orthogonal decomposition if and only if for all $i > j$, $A_i \in \mc{A}_i$ and $A_j \in \mc{A}_j$, the mapping complex $\mc{A}(A_i, A_j)$ is acyclic, and each $A \in \mc{A}$ admits a filtration
\begin{equation*}
    0 = A_{\geq n+1} \lra A_{\geq n} \lra \dots \lra A_{\geq 2} \lra A_{\geq 1} = A
\end{equation*}
with $A_{\geq i} \in \mc{A}_{\geq i}$ for each $1 \leq i \leq n+1$, where the cone of $A_{\geq i+1} \ra A_{\geq i}$ lies in $\mc{A}_i$. 
\end{remark}

The following prototypical example of a semi-orthogonal decomposition was given by Bernstein, Gelfand and Gelfand in \cite{bernstein78}. 

\begin{example} \label{def:semiorthogonaldecompositionprojectivespace}
Let $S$ be a scheme and let $X = \P^n_S$ be the projective space over $S$.
Let $\mc{A} = \Perf(X)$ be the pretriangulated dg category of perfect complexes of $\OO_X$-modules.
For $i \in \Z$, define $\mc{A}_i = \langle \OO_X(i) \rangle$ to be the full pretriangulated dg subcategory of $\mc{A}$ generated by the line bundle $\OO_X(i)$. 
Then $\mc{A}(\OO_X(i), \OO_X(j))$ is acyclic for all $j+n \geq i > j$, which follows from the fact that the sheaf cohomology $H^*(X, \OO_X(-r))$ vanishes for $r=1, \dots n$. 
Furthermore, by a dg version of \cite[lemma 3.5.2]{schlichting11}, $\mc{A}$ is generated by the line bundles $\OO_X(i)$ with $i \leq 0$.
The acyclic Koszul complex associated with the surjection $\OO_X^{\oplus n+1} \ra \OO_X$ yields an expression of $\OO_X(-n-1)$ in terms of $\OO_X, \dots, \OO_X(-n)$.
Hence $\langle \mc{A}_{-n}, \dots, \mc{A}_0 \rangle$ is a semi-orthogonal decomposition of $\mc{A}$.
By Corollary \ref{corollary:twistingsemiorthogonaldecompositionbyequivalence}, we can always twist a semi-orthogonal decomposition by a line bundle to obtain another semi-orthogonal decomposition.
In particular, $\langle \mc{A}_i, \dots, \mc{A}_{i+n} \rangle$ is a semi-orthogonal decomposition of $\mc{A}$ for each $i \in \Z$.
This trick will be useful once we start dealing with duality.
\end{example}

Here is another extension of the definition of a semi-orthogonal decomposition following \cite[definition 3.1]{scherotzke20} by allowing infinitely many components, which has applications in equivariant $\K$- and $\GW$-theory. 

\begin{definition} \label{def:infinitesemiorthogonaldecomposition}
Let $\mc{A}$ be a pretriangulated dg category.
For any pre-ordered set $(P,\leq)$ and a collection of full pretriangulated subcategories $\{\mc{A}_i \subset \mc{A} \mid i \in P\}$, the subcategories $\mc{A}_i$ form a \emph{pre-ordered semi-orthogonal decomposition of $\mc{A}$} if
\begin{enumerate}[label=(\roman*)]
    \item for all $i \in P$, the inclusion $H^0\mc{A}_i \subset H^0\mc{A}$ admits both a left and a right adjoint;
    \item for all $i, j \in P$ such that $i < j$, $\mc{A}_i \subset \ro{\mc{A}_j}$; and
    \item $\mc{A}$ is the smallest pretriangulated dg subcategory containing $\mc{A}_i$ for each $i \in P$.
\end{enumerate}
\end{definition}

\begin{example} \label{example:semiorthogonaldecompositionperfectcomplexesmultiplicativegroup}
Let $G = \GG_m$ be the multiplicative group over a base field $k$. 
The pretriangulated dg category $\mc{A} = \Perf^G(k)$ of $G$-equivariant perfect complexes over $k$ is equivalent to the bounded derived dg category of finite dimensional graded $k$-vector spaces.
For $i \in \Z$, let $k(i)$ be the graded vector space with $k(i)_j = \delta_{ij} k$, where $\delta_{ij}$ is the Kronecker delta, and let $\mc{A}_i = \langle k(i) \rangle$. 
Then $\mc{A}_i$ is equivalent to $\Perf(k)$.
Note that $\mc{A}(k(i),k(j)) = 0$ for all $i,j \in \Z$ such that $i \neq j$. 
The smallest pretriangulated dg category containing each $\mc{A}_i$ is $\mc{A}$ itself, and since we only consider finite dimensional $k$-vector spaces, each $V \in \mc{A}$ is contained in $\mc{A}_{[i,j]}$ for some $i,j \in \Z$. 
Hence $\langle \mc{A}_i \mid i \in \Z \rangle$ is an example of a semi-orthogonal decomposition as in definition \ref{def:infinitesemiorthogonaldecomposition}.
\end{example}

A useful feature of semi-orthogonal decompositions is that they are stable under equivalences of pretriangulated dg categories.

\begin{proposition} \label{proposition:twistsemiorthogonaldecompositionbyequivalence}
Let $F: \mc{A} \ra \mc{B}$ be an equivalence of pretriangulated dg categories and let $\langle \mc{A}_-, \mc{A}_+ \rangle$ be a semi-orthogonal decomposition of $\mc{A}$. 
Then the image $\langle F(\mc{A}_-), F(\mc{A}_+) \rangle$ is a semi-orthogonal decomposition of $\mc{B}$. 
\end{proposition}

\begin{proof}
Since $F$ preserves shifts and cones, $F(\mc{A}_-)$ and $F(\mc{A}_+)$ are full pretriangulated dg subcategories of $\mc{B}$. 

For $A_- \in \mc{A}_-$ and $A_+ \in \mc{A}_+$, as $F$ is fully faithful, $\mc{B}(F(A_+), F(A_-)) = \mc{A}(A_+, A_-)$, which is acyclic by assumption.
Let $B \in \mc{B}$ and let $A \in \mc{A}$ such that $F(A) \cong B$, which exists since $F$ is essentially surjective. 
Then $A$ admits a sequence
\begin{equation*}
    A_+ \stlra{f} A \lra \cone(f)
\end{equation*}
with $A_+ \in \mc{A}_+$ and $\cone(f) \in \mc{A}_-$ by assumption.
Taking the image of this sequence under $F$ yields a sequence
\begin{equation*}
    F(A_+) \stlra{F(f)} F(A) \lra F(\cone(f)),
\end{equation*}
where $F(\cone(f)) \cong \cone(F(f))$.
As $F(A) \cong B$, there is a sequence
\begin{equation*}
    F(A_+) \stlra{g} B \lra \cone(g)
\end{equation*}
with $\cone(g) \in F(\mc{A}_-)$. 
Therefore, $\langle F(\mc{A}_-), F(\mc{A}_+) \rangle$ is a semi-orthogonal decomposition of $\mc{B}$, as was to be shown.
\end{proof}

\begin{remark} \label{remark:quasiequivalencepreservessod}
We could weaken the conditions on $F$ in Proposition \ref{proposition:twistsemiorthogonaldecompositionbyequivalence} by only requiring that $F$ be a quasi-equivalence, so that the induced functor on homotopy categories 

fully faithful, so that the natural map $\mc{A}(A, A') \ra \mc{B}(F(A), F(A'))$ is a quasi-isomorphism for all $A, A' \in \mc{A}$, and quasi-essentially surjective, so that the induced
\end{remark}

\begin{corollary} \label{corollary:twistingsemiorthogonaldecompositionbyequivalence}
Let $\mc{A}$ be a pretriangulated dg category with a semi-orthogonal decomposition $\langle \mc{A}_i \mid i \in \Z \rangle$. 
Let $F: \mc{A} \ra \mc{A}$ be an equivalence of dg categories. 
Then $\langle F(\mc{A}_i) \mid i \in \Z \rangle$ is a semi-orthogonal decomposition of $\mc{A}$. 
\end{corollary}

\begin{proof}
By proposition \ref{proposition:twistsemiorthogonaldecompositionbyequivalence}, $\langle F(\mc{A}_{\leq i}), F(\mc{A}_{> i}) \rangle$ is a semi-orthogonal decomposition of $\mc{A}$ for each $i \in \Z$. 
By assumption, each $A \in \mc{A}$ is contained in $\mc{A}_{[i,j]}$ for some $i,j \in \Z$, so each $F(A)$ is contained in $F(\mc{A}_{[i,j]}) = (F\mc{A})_{[i,j]}$ for some $i,j \in \Z$. 
Since $F$ is an equivalence, each $A \in \mc{A}$ is isomorphic to $F(B)$ for some $B \in \mc{A}$, and the result follows.
\end{proof}

\printbibliography

\end{document}